\documentclass[preprint,12pt]{elsarticle}

\journal{Discrete Applied Mathematics}

\usepackage{amsmath,amsthm}
\usepackage{amssymb}
\usepackage[utf8]{inputenc}
\usepackage[english]{babel}
\usepackage{enumitem}
\usepackage{mathtools}
\usepackage{verbatim}
\usepackage{parskip}
\usepackage{float}
\usepackage{tikz}
\usepackage{url}

\usepackage{pgfplots}

\tikzstyle{vertex}=[draw,thick,fill=white,circle,inner sep=2pt]

\newtheorem{theorem}{Theorem}
\newtheorem{lemma}[theorem]{Lemma}

\newtheorem{definition}[theorem]{Definition}

\usepackage[margin=1.1in]{geometry}
\usepackage{setspace}
\usepackage[utf8]{inputenc}

\usepackage{lipsum}
\usepackage{amsfonts}
\usepackage{graphicx}
\usepackage{epstopdf}
\usepackage{algorithmic}
\ifpdf
  \DeclareGraphicsExtensions{.eps,.pdf,.png,.jpg}
\else
  \DeclareGraphicsExtensions{.eps}
\fi

\usepackage{amsopn}

\begin{document}

\begin{frontmatter}

\title{Constructing bounded degree graphs with prescribed degree and neighbor degree sequences}

\author{Uro\v{s} \v{C}ibej$^{1}$ and Aaron Li$^{2}$ and Istv\'an Mikl\'os$^{3,4}$ and Sohaib Nasir$^{5}$ and Varun Srikanth$^{6}$}
\address{$^{1}$University of Ljubljana, Faculty of Computer and Information Science, Ve\v{c}na pot 113, 1000 Ljubljana, Slovenia\\
$^2$University of Minnesota, 206 Church St. SE, Minneapolis, MN, USA, 55455\\
$^3$R\'enyi Institute, ELKH, 1053 Budapest, Re\'altanoda u. 13-15, Hungary \\
$^4$SZTAKI, ELKH, 1111 Budapest, L\'agym\'anyosi u. 11, Hungary\\
$^5$Vassar College, 124 Raymond Avenue, Poughkeepsie, New York 12604\\
$^6$Georgia Institute of Technology
North Avenue, Atlanta, GA 30332}
    
\begin{abstract}
Let $D = d_1, d_2, \ldots, d_n$ and $F = f_1, f_2,\ldots, f_n$ be two sequences of positive integers. We consider the following decision problems: is there a $i)$ multigraph, $ii)$ loopless multigraph, $iii)$ simple graph, $iv)$ connected simple graph, $v)$ tree, $vi)$ caterpillar $G = (V,E)$ such that for all $k$, $d(v_k) = d_k$ and $\sum_{w\in \mathcal{N}(v_k)} d(w) = f_k$ ($d(v)$ is the degree of $v$ and $\mathcal{N}(v)$ is the set of neighbors of $v$). Here we show that all these decision problems can be solved in polynomial time if $\max_{k} d_k$ is bounded.

The problem is motivated by NMR spectroscopy of hydrocarbons. 
\end{abstract}

\begin{keyword}
  Degree sequences \sep Neighbor degree constraint \sep NMR spectroscopy 
  
  \MSC 05C05 \sep \MSC 05C07 \sep \MSC 05C40 \sep \MSC 05C85 \sep \MSC 05C92 \sep \MSC 92E10
\end{keyword}

\end{frontmatter}
\maketitle

  

\section{Introduction}


In statistical testing of networks, a network from real life must be compared with a background distribution of random graphs. In such null models of random graphs, the degree sequence of the graphs is typically fixed. However, there are structural properties that are not preserved by the degree sequence. One example for such a property is the \emph{assortativity} \cite{stantonpinar}. A network is called \emph{assortative} if vertices tend to be connected to vertices with similar degrees, and it is called \emph{dissortative} if low degree vertices tend to connect to high degree vertices. The observation that real life networks with similar degree sequences might have different assortativity urged research to develop null models that preserve structural properties above the degree sequence. Such models are, for example, the joint degree matrix model \cite{stantonpinar,czabarkaetal} that prescribes the number of edges between vertex classes with given degrees and the \emph{$dk$-random graphs} that considers the distribution of induced sub-graphs up to a given size \cite{orsinietal}. While there are polynomial algorithms to construct a realization of a joint degree matrix \cite{czabarkaetal}, constructing a graph with prescribed small sub-graph distribution appears to be a hard problem. This motivated the question: ``What kind of local properties make the graph construction hard?" Erd{\H o}s and Mikl\'os showed that it is already NP-complete to ask if there exists a graph with given degree and neighbor degree sequence \cite{em2018}.

In this paper, we show that there exists a polynomial algorithm to construct graphs with prescribed degree and neighbor degree sequence if there is a bound on the maximum degree. We consider several variants of the problem: we might require that the graph be a simple graph, connected simple graph, tree, or caterpillar. Particular emphases are on the cases when the graph is connected as it also has an application of obtaining the chemical structure of hydrocarbons by NMR spectroscopy data.


Hydrocarbons are the simplest organic molecules consisting of only hydrogen and carbon atoms. A carbon atom makes four covalent bonds, while a hydrogen atom makes only one covalent bond. What follows is that hydrogen atoms can bond only to carbon atoms in hydrocarbons. Therefore, the information on the covalent bonds between carbon atoms completely describes the chemical structure of a hydrocarbon molecule. Indeed, if a carbon atom makes $k$ covalent bonds to other carbon atoms, then it must make $4-k$ bonds to hydrogen atoms. The diagram representing only the covalent bonds between the carbon atoms is a connected graph with maximum degree $4$. This graph is called the skeletal structure. A hydrocarbon might contain multiple bonds between two carbon atoms. Such a hydrocarbon is called unsaturated, and the skeletal structure is a connected, loopless multigraph in that case. When there are only single bonds between carbon atoms, the hydrocarbon is called saturated, and then the skeletal structure is a connected simple graph.

The chemical formula $C_nH_m$ means that there are $n$ carbon atoms and $m$ hydrogen atoms in a molecule. There are simple chemical measurements to obtain $n$ and $m$. When $m= 2n+2$, the skeletal structure is a tree. Indeed, in that case, there are $n-1$ covalent bonds between carbon atoms, and any connected graph with $n$ vertices and $n-1$ edges is a tree. As $m$ decreases, the number of covalent bonds between carbon atoms increases. This might be obtained by multiple bonds between carbon atoms (that is, we are talking about unsaturated hydrocarbons) or making cycles in the skeletal structure. Unsaturated hydrocarbons react with halogens while saturated hydrocarbons do not react, and that simple chemical reaction helps separate unsaturated hydrocarbons from saturated hydrocarbons with cycles in their skeletal structure. In conclusion, the chemical formula $C_nH_m$ is easy to obtain and it is easy to decide if the hydrocarbon is saturated or not as well.

More information on the chemical structure of a hydrocarbon can be attained from NMR spectroscopy. Roughly speaking, the position (frequency) of a peak in the NMR spectroscopy depends on the number of hydrogen atoms bonding to a particular carbon atom causing the peak, and the size of the peak tells the number of such carbon atoms. Furthermore, a peak (roughly, a Gaussian bell-shaped curve) is split depending on the number of hydrogen atoms bonding to neighbor carbon atoms. A peak is split into $l+1$ parts when there are altogether $l$ hydrogen atoms on the neighbor carbon atoms. From this information, the degree of a carbon atom as well as the sum of the degrees of the neighbor carbon atoms in the skeletal structure can be obtained. Indeed, if a peak related to $k$ hydrogen-carbon bonds is split into $l+1$ parts, then the degree of the carbon atom is $4-k$, and the sum of the degrees of the neighbor carbon atoms is $4\times(4-k) - l$.

This raises the following graph theoretical question: given degree sequences $D = d_1, d_2, \ldots, d_n$ and $F = f_1, f_2, \ldots, f_n$, is there a graph $G = (V,E)$, such that for all $k$, $d(v_k) =d_k$ and $\sum_{w\in \mathcal{N}(v_k)} d(w) = f_k$, where $d(v)$ is the degree of vertex $v$ and $\mathcal{N}(v)$ is the set of the neighbors of $v$? Erd{\H o}s and Mikl\'os showed that this decision problem is NP-complete in general \cite{em2018}. Here we consider a specific case: the maximum degree is $4$, and we are looking for connected realizations. We show that this specific case can be solved in polynomial time.

There is typically more than one graph with a given degree and neighbor degree sequence. Some of the solutions are chemically not stable due to spherical constraints. For example, it is known that the hydrocarbon whose skeletal structure would be the balanced unrooted ternary tree with $17 ( = 1 + 4 + 12)$ vertices is chemically not stable. In principle, molecules that have skeletal structure with larger diameter tend to be more stable. Among trees with given degree sequence, caterpillars have the maximum diameter. This motivates us to ask when a caterpillar with given degree and neighbor degree sequence exists.

The paper is structured as follows. In the Preliminaries, we give some definitions on graphs and also introduce the concept of labeled stub-stars. When each edge in a graph is cut into two half-edges (stubs), we get stub-stars. If each stub is labeled by the degree of the neighbor vertex incident to the edge whose cut resulted in the stub, the sum of these labels is the sum of the neighbor degrees.  In the next section, we give necessary and sufficient conditions when an ensemble of labeled stub-stars has graph realizations with certain properties. Our main interest is in the cases when the graph is a connected, simple graph or a tree or a caterpillar. However, for the sake of completeness, we also discuss the cases when the graph is a multigraph, a loopless multigraph, or a simple graph. After this, we show that for each of the six cases (multigraph, loopless multigraph, simple graph, simple connected graph, tree, caterpillar) a polynomial time solvable integer programming feasibility problem exists to find an ensemble of stub-stars that is consistent with a given degree and neighbor degree sequence and also satisfies the necessary conditions to have a graph realization with the given properties.
%

\section{Preliminaries}

First, we give some graph-theoretical definitions.

\begin{definition}
A graph $G = (V,E)$ is a  \emph{multigraph} if $E$ is a multiset of $V\times V$. $G$ is a \emph{loopless multigraph} if $E$ is a multiset of ${V\choose 2}$. $G$ is a \emph{simple graph} if $E$ is a subset of ${V\choose 2}$.

A  \emph{caterpillar} is a tree in which the non-leaf vertices form a path.
\end{definition}

\begin{definition}
A degree sequence $\{d_1, d_2,\ldots, d_n\}$ is \emph{graphical} if there exists a vertex labeled graph $G = (V,E)$ such that for each $v_i$, $d(v_i) = d_i$.

A pair of degree sequences (or \emph {bipartite degree sequence}) $(\{d_{1,1}, d_{1,2},\ldots, d_{1,n}\}\{d_{2,1}, d_{2,2},\ldots,d_{2,m}\})$ is \emph{graphical} if there exists a simple bipartite graph $G  = (U,V,E)$, such that for each $u_i$, $d_{u_i} = d_{1,i}$, and for each $v_j$, $d(v_j) = d_{2,j}$.

A degree sequence is \emph{forest realizable} if it has a forest realization.
\end{definition}

\begin{definition}
Let $D= d_1, d_2,\ldots d_n$ and $F = f_1, f_2, \ldots, f_n$ be a pair of degree and neighbor degree sequences. We say that $G = (V,E)$ is a \emph{realization} of $(D,F)$ if for all $k$, $d(v_k) = d_k$ and $\sum_{w\in \mathcal{N}(v_k)} d(w) = f_k$, where $\mathcal{N}(v)$ denotes the set of neighbors of $v$.
\end{definition}

We now introduce the main technical concept used in this paper: the labeled stub-stars.

\begin{definition}
A \emph{labeled stub-star} is a vertex with $d$ stubs (half edges). Each stub $s_k$ is labeled with $d_k$. Let $\mathcal{S} = \{S_1, S_2, \ldots, S_n\}$ be a multiset (we will also call it \emph{ensemble}) of labeled stub-stars. A vertex labeled  graph $G = (V,E)$ is a \emph{realization} of $\mathcal{S}$ if for each $v_i$, the neighbors of $v_i$ have degrees $d_{i,1}, d_{i,2},\ldots, d_{i,k}$, where $d_{i,1}, d_{i,2},\ldots d_{i,k}$ are the labels of the stub-star $S_i$. We say that $\mathcal{S}$ is multigraph, loopless multigraph, simple graph, connected simple graph, tree, caterpillar realizable, respectively, if it has a realization which is a multigraph, loopless multigraph, simple graph, connected simple graph, tree, caterpillar, respectively.
\end{definition}

\begin{definition}
The \emph{chromatic degree $d_{(i,j)}$} of a labeled stub-star $S$ is $0$ if the degree of $S$ is neither $i$ nor $j$. If its degree is $i$, then $d_{(i,j)}(S)$ is the number of times $S$ contains label $j$. Finally, if the degree of $S$ is $j$, then $d_{(i,j)}(S)$ is the number of times $S$ contains label $i$.

For an ensemble of labeled stub-stars $\mathcal{S} \{S_1, S_2, \ldots, S_n\}$, we define a degree sequence 
$$
D_{i,i}(\mathcal{S}) := d_{(i,i)}(S_1), d_{(i,i)}(S_2), \ldots, d_{(i,i)}(S_n).
$$

We define the bipartite degree sequence $D_{i,j}(\mathcal{S})$ for all $i\ne j$ similarly. The chromatic degrees $d_{(i,j)}(S_k)$ are naturally arranged into the two degree sequences according to the degree of $S_k$ if it is either $i$ or $j$. For other degrees of $S_k$, the corresponding chromatic degree is $0$, and it can be arbitrarily added to any degree sequence in $D_{i,j}(\mathcal{S})$.
We can also talk about the \emph{monochromatic sub-graphs} in a realization of an ensemble of labeled stub-stars: the monochromatic sub-graph with color $(i,j)$ contains the edges that connects vertices with degree $i$ and $j$.
 \end{definition}

The labels of a labeled stub-star form a partition. Below we give the necessary definitions and notations for partitions used in this paper.
\begin{definition}
We will denote by $\lambda \dashv n$ that $\lambda = \{\lambda_1, \lambda_2, \ldots, \lambda_k\}$ is a partition of a positive integer $n$. That is, each $\lambda_i$ is a positive integer, and it holds that
$$
\sum_{i=1}^k \lambda_i = n.
$$
The \emph{height} of $\lambda = \{\lambda_1, \lambda_2, \ldots, \lambda_k\}$ is $k$, and is denoted by $h(\lambda)$. The \emph{size} of $\lambda$ is $n$, and is denoted by $|\lambda|$. Finally, $N(\lambda,j)$ denotes how many times $j$ appear in $\lambda$, that is
$$
N(\lambda,j) := |\{\lambda_i \in \lambda| \lambda_i = j\}|.
$$

\end{definition}

If $\lambda$ comes from the labels of a stub-star, then $h(\lambda)$ corresponds to the degree of the stub-star while $|\lambda|$ corresponds to the neighbor degree sum.

We are going to solve the graph realization problems defined above via integer programming feasibility problems that we are going to introduce now.
\begin{definition}
The \emph{integer programming feasibility problem} $(\mathcal{X},\mathcal{L})$ consists of a set of indeterminates, $\mathcal{X}$ and a linear inequality system $\mathcal{L}$ whose indeterminates are from $\mathcal{X}$ and
asks if there is an integer assignment to $\mathcal{X}$ that satisfies $\mathcal{L}$.
\end{definition}

Some of the inequalities we are going to use contain the function $\min\{x,y\}$ or $\max\{x,y\}$. These can be expressed as linear inequality systems if an upper bound for $|x-y|$ is available. The following lemma is well-known in operational research, but for the sake of completeness, we give a proof here.

\begin{lemma}\label{lem:express-min-max}
Let $x$, $y$ and $M$ be integer numbers satisfying
$$
M   \ge |x-y|,
$$ and let $z$ and $b$ be indeterminates. Then the following inequality system
\begin{eqnarray}
z &\le & x \label{eq:lemma2.1.1}\\
z &\le & y \label{eq:lemma2.1.2}\\
z & \ge & x - Mb \label{eq:lemma2.1.3}\\
z & \ge & y- M(1-b)\label{eq:lemma2.1.4}\\
0 & \le & b\\
b & \le & 1 
\end{eqnarray}
has at most two integer solutions, and in every case $z = \min\{x,y\}$.

Similarly, the following inequality system
\begin{eqnarray}
z &\ge & x \\
z &\ge & y \\
z & \le & x + Mb \\
z & \le & y+ M(1-b)\\
0 & \le & b\\
b & \le & 1 
\end{eqnarray}
has at most two integer solutions, and in every case $z = \max\{x,y\}$.

\end{lemma}
\begin{proof}
We prove the $\min\{x,y\}$ case. Observe that $b$ is either $0$ or $1$. If $x < y$, then there is no solution with $b = 1$ since the inequalities in equations~\ref{eq:lemma2.1.1}~and~\ref{eq:lemma2.1.4} would contradict to each other. On the other hand, $b = 0$ and $z = x$ is a solution, and the only solution since $x\le z \le x$ must hold.

Similarly, when $y < x$, the only solution is $z = y$, $b = 1$. Finally, if $x=y$, then $z = x ( = y)$ and $b \in\{0,1\}$.

The proof for $\max\{x,y\}$ goes analogously.
\end{proof}

\section{Realizations of a set of labeled stub-stars}

In this section, we state and prove theorems on realizations of labeled stub-stars. We need the well-known Erd{\H o}s-Gallai and Gale-Ryser inequalities.

\begin{theorem}[Erd{\H o}s-Gallai, \cite{eg}]
Let $D = d_1\ge d_2 \ge \ldots \ge d_n$ be a degree sequence. Then $D$ is simple graph realizable if and only if the sum of the degrees is even, and for all $k$, the inequality
\begin{equation}
\sum_{g=1}^k d_g \le k(k-1) + \sum_{h=k+1}^n \min\{k,d_h\} \label{eq:eg}
\end{equation}
holds.
\end{theorem}

\begin{theorem}[Gale-Ryser, \cite{gale,ryser}]
Let $D = (\{d_{1,1}\le d_{1,2}\le \ldots \le d_{1,n}\},\{d_{2,1}, d_{2,2},\ldots,d_{2,m}\})$ be a bipartite degree sequence. Then $D$ is bipartite graph realizable if and only if
$$
\sum_{g=1}^n d_{1,g} = \sum_{h=1}^m d_{2,h}
$$
and for each $k$, the inequality
\begin{equation}
\sum_{i=g}^k d_{1,g} \le \sum_{h=1}^m \min\{k,d_{2,h}\} \label{eq:gr}
\end{equation}
holds.
\end{theorem}

When the degree sequences have a maximum degree $\Delta$, then only the first $\Delta$ Erd{\H o}s-Gallai and the first $\Delta-1$ Gale-Ryser inequalities have to be checked as the following lemma states.

\begin{lemma}
Let $D = d_1\ge d_2 \ge \ldots \ge d_n$ be a degree sequence. Then for each $k > d_1$, the inequality in equation~(\ref{eq:eg}) holds. Similarly, let $D = (\{d_{1,1}\ge d_{1,2}\ge \ldots \ge d_{1,n}\},\{d_{2,1}, d_{2,2},\ldots,d_{2,m}\})$ be a bipartite degree sequence. Then for all $k\ge \max_{j}\{d_{2,j}\}$, the inequality in equation~(\ref{eq:gr}) holds.
\end{lemma}
\begin{proof}
First we consider the Erd{\H o}s-Gallai inequalities. If $k > d_1$, then it holds that
$$
\sum_{i=1}^k d_i \le k\times d_1 \le k(k-1) \le k(k-1) + \sum_{j=k+1}^n \min\{k,d_j\}.
$$
Regarding the Gale-Ryser inequalities, if $k \ge \max_{j}\{d_{2,j}\}$, then it holds that
$$
\sum_{i=1}^k d_{1,i} \le \sum_{j=1}^n d_{2,j} \le \sum_{j=1}^m \min\{k,d_{2,j}\}.
$$
\end{proof}

\begin{theorem}\label{theo:multi-simple-realizations}
Let $\mathcal{S} = S_1, S_2, \ldots, S_n$ be an ensemble of labeled stub-stars with maximum degree $\Delta$.
Let $x_{\lambda}$ denote the number of labeled stub-stars in $\mathcal{S}$ whose labels form the partition $\lambda$. Then
$\mathcal{S}$ is
\begin{enumerate}[label=(\alph*)]
\item multigraph
\item loopless multigraph
\item simple graph
\end{enumerate}
realizable if and only if 
\begin{enumerate}[label=(\alph*)]
\item \label{case:a} for each $i$, 
\begin{equation}
    \sum_{\lambda | h(\lambda) = i} x_{\lambda} N(\lambda, i) \equiv 0 \,\,\,\,\, \pmod{2}\label{eq:multigraph-condition}
\end{equation}
and for each $i\neq j$, 
\begin{equation}
    \sum_{\lambda | h(\lambda) = i} x_{\lambda}N(\lambda,j) = \sum_{\lambda' | h(\lambda') = j} x_{\lambda'}N(\lambda',i).\label{eq:bipartite-multigraph-condition}
\end{equation}
\item the conditions in case~\ref{case:a} holds and for each $i$, 
\begin{equation}
2 \max_{\lambda | h(\lambda) = i} \left\{\min\{1,x_{\lambda}\} N(\lambda,i)\right\} \le \sum_{\lambda'|h(\lambda') = i} x_{\lambda'} N(\lambda,i).\label{eq:loopless-multigraph-condition}
\end{equation}
\item the conditions in case~\ref{case:a} holds, for each $i$, the first $\Delta$ Erd{\H o}s-Gallai inequalities hold for $D_{i,i}(\mathcal{S})$, and for each $i\neq j$, the first $\Delta-1$ Gale-Ryser inequalities hold for the bipartite degree sequence $D_{i,j}(\mathcal{S})$.
\end{enumerate}
\end{theorem}

\begin{proof}
\ \ 

\begin{enumerate}[label=(\alph*)]
    \item \begin{sloppypar}
    It is well known that a degree sequence $\{d_1, d_2, \ldots, d_n\}$ has a multigraph realization if and only if $\sum_{i=1}^n d_i$ is even. Similarly, the bipartite degree sequence $(\{d_{1,1},d_{1,2},\ldots, d_{1,n}\}, \{d_{2,1},d_{2,2},\ldots,d_{2.m}\})$ has a bipartite multigraph realization if and only if $\sum_{i=1}^n d_{1,i} = \sum_{j=1}^m d_{2,j}$. Observe that equation~(\ref{eq:multigraph-condition}) says that the sum of the degrees in $D_{i,i}(\mathcal{S})$ is even, furthermore, equation~(\ref{eq:bipartite-multigraph-condition}) says that the sum of the degrees in the two vertex classes are the same in $D_{i,j}(\mathcal{S})$.
    \end{sloppypar}

    Let $G$ be a multigraph realization of $\mathcal{S}$. Then for each $i$, the monochromatic subgraph with color $(i,i)$ is a multigraph, and for each $i\ne j$, the monochromatic subgraph with color $(i,j)$ is a bipartite multigraph. That is, equations~(\ref{eq:multigraph-condition})~and~(\ref{eq:bipartite-multigraph-condition}) hold. 
    
    On the other hand, if equations~(\ref{eq:multigraph-condition})~and~(\ref{eq:bipartite-multigraph-condition}) hold, then there are multigraph realizations for each $D_{i,i}(\mathcal{S})$ and there are bipartite multigraph realizations for each $D_{i,j}(\mathcal{S})$. The union of them is a multigraph realization of $\mathcal{S}$.
    
    \item It is also well known that a degree sequence $\{d_1, d_2,\ldots, d_n\}$ has a loopless multigraph realization if and only if the sum of the degrees is even and 
    $$
    \max_i\{d_i\} \le \sum_{i=1}^n d_i - \max_i\{d_i\}
    $$
    Observe that equation~(\ref{eq:loopless-multigraph-condition}) describes this condition for the degree sequence $D_{i,i}(\mathcal{S})$.
    
    Let $G$ be a loopless multigraph realization of $\mathcal{S}$. Then for each $i$, the monochromatic subgraph with color $(i,i)$ is a loopless multigraph, and for each $i\ne j$, the monochromatic subgraph with color $(i,j)$ is a bipartite (loopless) multigraph. (Observe that bipartite graphs cannot contain a loop.) That is, equations~(\ref{eq:multigraph-condition}),~(\ref{eq:bipartite-multigraph-condition})~and~~(\ref{eq:loopless-multigraph-condition}) hold. 
    
    On the other hand, if equations~(\ref{eq:multigraph-condition}),~(\ref{eq:bipartite-multigraph-condition})~and~~(\ref{eq:loopless-multigraph-condition}) hold, then there are loopless multigraph realizations for each $D_{i,i}(\mathcal{S})$ and there are bipartite multigraph realizations for each $D_{i,j}(\mathcal{S})$. The union of them is a loopless multigraph realization of $\mathcal{S}$.

    \item Let $G$ be a simple graph realization of $\mathcal{S}$. Then for each color $(i,i)$, the monochromatic subgraph is a simple realization of the degree sequence $D_{i,i}(\mathcal{S})$, and for each color $(i,j)$, $i\ne j$, the monochromatic subgraph is a simple bipartite realization of the bipartite degree sequence $D_{i,j}(\mathcal{S})$. Therefore the conditions in case~\ref{case:a} as well as the first $\Delta$ Erd{\H o}s-Gallai and the first $\Delta-1$ Gale-Ryser inequalities hold.
    
    On the other hand, if the conditions in case~\ref{case:a} as well as the first $\Delta$ Erd{\H o}s-Gallai and the first $\Delta-1$ Gale-Ryser inequalities hold, then there are monochromatic simple graph realizations of each color $(i,i)$, and there are monochromatic simple bipartite graph realizations for each color $(i,j)$, $i\ne j$. Observe that the union of them is a simple graph realization of $\mathcal{S}$. Indeed, no parallel edges are possible in the union of these monochromatic graphs.
\end{enumerate}
\end{proof}

We now consider when an ensemble of labeled stub-stars has a forest realization. Given the previous theorem, we might hope that a sufficient condition is if for each $i \le j$, $D_{i,j}$ is a forest realizable degree sequence. 
However, this is not the case, as the following example demonstrates.

The example ensemble of labeled stub-stars $\mathcal{S}$ consists of four labeled stub-stars, $S_1, S_2, S_3, S_4$, respectively, with labels forming partitions
$(3,2), (3,2), (3), (2,2,1)$, respectively.


In this case, each $D_{i,i}(\mathcal{S})$ and $D_{i,j}(\mathcal{S})$ is forest realizable, but $\mathcal{S}$ has only one realization, shown below.

 
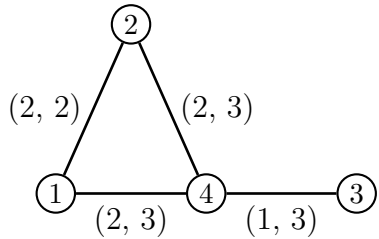
\begin{center}
     
    \begin{tikzpicture}[scale = 0.5]
    \node[vertex](1) at (0,0) {\small 1};
    \node[vertex](2) at (2,4.5) {\small 2};
    \node[vertex](4) at (4,0) {\small 4};
    \node[vertex](3) at (8,0) {\small 3};
    \draw[line width = 1pt] (1) -- node[left] {(2, 2)}  (2);
    \draw[line width = 1pt] (1) -- node[below] {(2, 3)} (4) -- node[right] {(2, 3)} (2);
     \draw[line width = 1pt] (4) -- node[below] {(1, 3)}(3);
    \end{tikzpicture}
    
 \end{center} 

The issue is that unlike loops and parallel edges, cycles can occur among multiple colors, so checking each individual degree sequence
is not sufficient. This motivates the next definition.

\begin{definition}
Let $\mathcal{S} = S_1, S_2, \ldots, S_n$ be an ensemble of labeled stub-stars with maximum degree $\Delta$. For any subset $I \subseteq \{(i,i)| 1\le i\le \Delta\}\cup {\Delta \choose 2}$, we define the \emph{multichromatic degree} of $S_l$ to be $d_{I}(S_l) := \sum_{(i,j) \in I} d_{(i,j)}(S_l)$ and the \emph{multichromatic degree sequence} to be $D_{I} := d_{I}(S_1), d_{I}(S_2), \ldots, d_I(S_n)$.
\end{definition}

For example, let $I = \{(2, 2), (2, 3)\}$ for the previous example $\mathcal{S}$. Then $d_{I}(S_1) = 2, d_{I}(S_2) = 2, d_{I}(S_3) = 0,$ and $d_{I}(S_4) = 2$.

\begin{theorem}\label{theo:tree-realizations}
Let $\mathcal{S} = S_1, S_2, \ldots, S_n$ be an ensemble of labeled stub-stars with maximum degree $\Delta$. $\mathcal{S}$ is forest realizable if and only if it is multigraph realizable and for any subset $I \subseteq \{(i,i)| 1\le i\le \Delta\}\cup {\Delta \choose 2}$, the multichromatic degree sequence $(d_{I}(S_1), d_{I}(S_2), \ldots, d_{I}(S_n))$ is forest realizable.
\end{theorem}

\begin{proof} 
We first show the conditions are necessary. Let $G$ be a forest realization of $\mathcal{S}$. Then $G$ is also a multigraph realization, so $\mathcal{S}$ is multigraph realizable. For any subset of colors $I$, we define the color-induced subgraph of $G$, denoted $G[I]$, to be the subgraph consisting of all vertices of $G$ along with all edges labeled with some element of $I$. Since $G$ is a forest, then $G[I]$ is also a forest. The degree sequence of $G[I]$ is $(m_{I, 1}, m_{I, 2}, \ldots, m_{I, n})$, so this sequence must be forest realizable.

To show the conditions are sufficient, we consider a graph realization $G$ which is minimal among all realizations of $\mathcal{S}$ in the number of components.
It is easy to see that such realization exists. Indeed, since for each color the corresponding degree sequence is forest realizable, it is multigraph realizable (actually, simple graph realizable, too). Any union of monochromatic simple graphs is a simple graph. As multigraph realizations exist, and there are a finite number of them, there exists a realization with minimal number of components.

Assume for contradiction that $G$ has a cycle. Then we will produce another realization $G'$ with one fewer components. First we construct a subgraph $H_k$, then we will find a series of swap operations that lead to $G'$ with one fewer components.

Let $C = C_0$ be a cycle, and let $I_0$ be the set of colors of its edges. Let $H_0$ be $G[I_0]$. $H_0$ cannot be one component as it would contradict that $D_{I_0}$ is forest realizable. However, it can happen that $H_0$ is in one component of $G$. Therefore while $H_i$ is in one component of $G$, we construct an $H_{i+1}$ in the following way:
\begin{enumerate}
    \item For all pair of edges $(e_1,e_2)$ such that $e_1$ and $e_2$ have the same color, $e_1$ is in $C_i$ and $e_2$ is in another component of $H_i$, construct a path from $e_1$ to $e_2$. For each edge $f$ in that path whose color is not in $I_i$, add its color to $I_{i+1}$, and label $f$ by $(e_1,e_2)$. An edge might be in several paths for several $(e_1,e_2)$ pairs, in that case take any of these labels.
    \item Add all the colors in $I_i$ to $I_{i+1}$. Then $H_{i+1}$ is defined as $G[I_{i+1}]$ and $C_{i+1}$ is defined as the component of $H_{i+1}$ that contains $C_i$.
\end{enumerate}
In finite number of steps, this procedure will create an $H_k$ which spans in more than one component of $G$. Indeed, for all $i = 0, 1, \ldots, k-1$, $I_i \subset I_{i+1}$, and $G$ cannot be a connected graph as it contains a cycle and is forest realizable.

Now we are going to construct a $G'$ that has one fewer components than $G$. Let $f_{1,k}$ and $f_{2,k}$ be two edges with the same color in $I_k \setminus I_{k-1}$ such that $f_{1,k}$ is in $C_k$ and $f_{2,k}$ is not in the same component in $G$ than $f_{1,k}$. Assume that $f_{1,k} = (v_1,v_2)$, $f_2 = (v_3,v_4)$, and $v_1$ and $v_4$ have the same degree. Furthermore, let us call the two components containing $f_{1,k}$ and $f_{2,k}$ by $G_1$ and $G_2$.

Now we remove $f_{1,k}$ and $f_{2,k}$ from $G$ and add edges $f_1' = (v_1,v_3)$ and $f_2' = (v_2,v_4)$. This \emph{swap operation} creates another simple graph realization of $\mathcal{S}$ from $G$. If $f_{1,k}$ or $f_{2,k}$ is in a cycle in $G$, then the swap operation merges the two components of $G$ into one component, thus, we arrive to a $G'$ realization with one less component. Otherwise the swap operation does not change the number of components. Indeed, in that case, both removing $f_{1,k}$ and $f_{2,k}$ splits their components in $G$ into two, however, the two new edges connects two pairs.

On the other hand, consider the label $(f_{1,k-1}, f_{2,k-1})$ of $f_{1,k}$ if the swap operation does not decrease the number of components. We claim the following: if the swap operation does not decrease the number of components in $G$ then $f_{1,k-1}$ and $f_{2,k-1}$ are in different components of $G$ after the swap operation. Indeed, we know that there was a path between $f_{1,k-1}$ and $f_{2,k-1}$ before the swap operation that contained $f_{1,k}$. If there is a path between $f_{1,k-1}$ and $f_{2,k-1}$ after the swap operation, then either this path is in $G_1$, and in that case $f_{1,k}$ was in a cycle, a contradiction, or this path contains edges in $G_2$, but in that case $f_{2,k}$ was in a cycle, also a contradiction (or both).

Now we can also perform a swap operation using edges $f_{1,k-1}$ and $f_{2,k-1}$. This operation either decreases the number of components, and then we arrive to $G'$ with one fewer components, or separates $f_{1,k-2}$ and $f_{2,k-2}$, the two edges in the label of $f_{1,k-1}$. We can iterate this process since all the swap operations separating $f_{1,i}$ and $f_{2,i}$ use edges with color in $I_k \setminus I_i$, and thus do not change any edge in the paths between $f_{j,1}$ and $f_{j,2}$ for any $j < i$.

Eventually, one of the swap operations will merge two components in $G$ since $f_{0,1}$ is in the cycle $C$.

We arrived to a contradiction that $G$ had the minimum number of components. Therefore, $G$ does not have a cycle, thus it is a forest realization.
\end{proof}

Observe that the proof also provides an algorithmic construction. Take any union of monochromatic forest realizations and if it contains a loop, then the construction in the proof provides a series of swap operation decreasing the number of components in the realization. Eventually, after a finite number of swap operations, a forest realization is constructed.

Finally, we give the necessary and sufficient conditions when an ensemble of labelled stub-stars are caterpillar realizable.
\begin{theorem}\label{theo:caterpillar-realizations}
Let $\mathcal{S} = S_1, S_2, \ldots, S_n$ be an ensemble of labeled stub-stars. Then $\mathcal{S}$ has a caterpillar realization if it is forest realizable, the sum of the degrees is $2n-2$, and there is no labelled stub-star in $\mathcal{S}$ with more than two labels larger than $1$.
\end{theorem}
\begin{proof}
First we show that the conditions are necessary. Let $G$ be a caterpillar realization of $\mathcal{S}$. Then $\mathcal{S}$ is clearly forest-realizable, and the sum of the degrees of the labelled stub-stars is $2n-2$. Also, in a caterpillar realization, any vertex has at most two neighbors with degree larger than $1$.

Now we show that the conditions are sufficient. If $\mathcal{S}$ is forest realizable, take any forest realization of it, $G$. We show that $G$ is a caterpillar. First, $G$ is a tree since the sum of the degrees is $2n-2$. Furthermore, it is a caterpillar since there is no vertex with more than two neighbors whose degrees are larger than $1$.
\end{proof}

\section{Degree and neighbor degree constraint realizations}

In the previous section, we gave necessary and sufficient conditions when an ensemble of labeled stub-stars have realizations with given properties. In this section, we show how those results fit into an integer programming feasibility problem. Recall that an integer programming feasibility problem asks if a bunch of linear inequalities have an integer solution. Although the integer programming feasibility problem in general is NP-complete, here we will have constant number of variables and equations.

First, we give the equation system that defines the possible ensembles of labeled stub-stars whose realizations are realizations of a given degree and neighbor degree sequence.

\begin{lemma}\label{lem:df-stubstar}
Let $D= d_1, d_2,\ldots d_n$ and $F = f_1, f_2, \ldots, f_n$ be a pair of degree and neighbor degree sequences. Then $G = (V,E)$ is a realization of $(D,F)$ if and only if $G$ is a realization of $\mathcal{S} = S_1, S_2, \ldots, S_n$ satisfying the following inequality system:
\begin{eqnarray}
\sum_{\lambda \dashv f\ \wedge\ h(\lambda) = d} x_{\lambda} 
&=& y_{d,f} \label{eq:df-stubstar} \\
x_{\lambda} & \ge & 0 \label{eq:df-stubstar-non-negativity}
\end{eqnarray}
where $x_{\lambda}$ is the number of labeled stub-stars in $\mathcal{S}$ with labels $\lambda$ and $y_{d,f}$ is the number of cases when $d_i=d$ and $f_i = f$.
\end{lemma}
\begin{proof}
Assume that $G = (V,E)$ is a realization of $(D,F)$. Then cutting each edge in $E$ and labeling the stubs of the so-emerging stub-stars by the degree of the neighbor vertices yields a labeled stub-star $S$ for each vertex $v$. The degree of $S$ is $d(v)$ and its labels form a partition of $f := \sum_{w\in\mathcal{N}(v)} d(w)$. Therefore any realization of $(D,F)$ is also a realization of an ensemble of labeled stub-stars that satisfies the equation system summarized in equation~\ref{eq:df-stubstar}.

Similarly, if $G = (V,E)$ is a realization of an ensemble of labeled stub-stars satisfying the equation system presented in equation~\ref{eq:df-stubstar}, then the degree sequence of $G$ is $D$, and the neighbor degree sum sequence is $F$.
\end{proof}

We can combine the equation system given in equation~\ref{eq:df-stubstar} with appropriate linear inequality systems describing the conditions (a)--(c) in Theorem~\ref{theo:multi-simple-realizations} and the conditions in Theorem~\ref{theo:tree-realizations}~and~\ref{theo:caterpillar-realizations}.

First, we start with multigraph realizations.

\begin{theorem}\label{theo:mg-realization}
Let $D= d_1, d_2,\ldots d_n$ and $F = f_1, f_2, \ldots, f_n$ be a pair of degree and neighbor degree sequences. Let $\Delta$ denote the largest degree in $D$. Then $(D,F)$ is multigraph realizable if and only if the integer programming feasibility problem $\{\mathcal{X}, \mathcal{L}\}$ has a solution, where $\mathcal{X}$ consists of all $x_{\lambda}$ variables for all partitions  $\lambda = \lambda_1, \lambda_2,\ldots, \lambda_{h(\lambda)}$ such that $h(\lambda) \le D$ and for all $i$, $\lambda_i \le \Delta$ and auxiliary variables $p_1, p_2, \ldots, p_{\Delta}$, and $\mathcal{L}$ consists of 
\begin{itemize}
    \item the inequality system given in equations~\ref{eq:df-stubstar}~and~\ref{eq:df-stubstar-non-negativity},
    \item for all $0< i< j \le \Delta$, the equation
 \begin{equation}  
     \sum_{\lambda| h(\lambda) = i} x_{\lambda} N(\lambda,j) = \sum_{\lambda'| h(\lambda') = j} x_{\lambda'} N(\lambda',i), \label{eq:mg-forced-bipartite}
 \end{equation}
     \item  and for all $i$, the equation
\begin{equation}
    \sum_{\lambda | h(\lambda) = i}x_{\lambda}N(\lambda, i) = 2p_i.\label{eq:mg-parity}
\end{equation}

\end{itemize}
\end{theorem}
\begin{proof}
Assume that $(D,F)$ has a multigraph realization $G$. Decompose $G$ into labeled stub-stars $\mathcal{S}$, then $G$ is a multigraph realization of $\mathcal{S}$.

Let $x_{\lambda}$ denote the number of labeled stub-stars in $\mathcal{S}$ whose labels are $\lambda$. Furthermore, let 
$$
p_i :=  \frac{1}{2} \sum_{\lambda| h(\lambda) = i} x_{\lambda} N(\lambda, i).
$$
We claim that these assignments form a solution to $(\mathcal{X},\mathcal{L})$. Indeed, the inequality system in equations~\ref{eq:df-stubstar}~and~\ref{eq:df-stubstar-non-negativity} holds due to Lemma~\ref{lem:df-stubstar}. Equations~\ref{eq:mg-forced-bipartite}~and~\ref{eq:mg-parity} hold due to Theorem~\ref{theo:multi-simple-realizations}. Indeed, equation~\ref{eq:mg-forced-bipartite} shows that for each $i\neq j$, the degree sequence $D_{i,j}(\mathcal{S})$ has a bipartite graph realization where all edges go between center of stub-stars with degree $i$ and $j$.
%
%
Furthermore, equation~\ref{eq:mg-parity} shows that for each $i$, the sum of the
degrees in $D_{i,i}(\mathcal{S})$ is even
Therefore, if $(D,F)$ has a multigraph realization, the integer programming feasibility problem $(\mathcal{X},\mathcal{L})$ has a solution.

Now assume that $(\mathcal{X},\mathcal{L})$ has a solution. Let $\mathcal{S}$ denote the ensemble of labeled stub-stars in which the number of labeled stub-stars with labels $\lambda$ is $x_{\lambda}$. Since equations~\ref{eq:mg-forced-bipartite}~and~\ref{eq:mg-parity} hold, $\mathcal{S}$ has a multigraph realization $G$ according to Theorem~\ref{theo:multi-simple-realizations}. This realization $G$ is a realization of $(D,F)$ by Lemma~\ref{lem:df-stubstar} since the inequality system given in equations~\ref{eq:df-stubstar}~and~\ref{eq:df-stubstar-non-negativity} holds.
 \end{proof}

For loopless multigraphs, a similar theorem can be proved.
\begin{theorem}
Let $D= d_1, d_2,\ldots d_n$ and $F = f_1, f_2, \ldots, f_n$ be a pair of degree and neighbor degree sequences. Let $\Delta$ denote the largest degree in $D$. Then $(D,F)$ is loopless multigraph realizable if and only if the integer programming feasibility problem $\{\mathcal{X}, \mathcal{L}\}$ has a solution, where $\mathcal{X}$ consists of all $x_{\lambda}$ variables for all partitions  $\lambda = \lambda_1, \lambda_2,\ldots, \lambda_{h(\lambda)}$ such that $h(\lambda) \le D$ and for all $i$, $\lambda_i \le \Delta$ and auxiliary variables $p_1, p_2, \ldots, p_{\Delta}, z_1, z_2, \ldots, z_r, b_1, b_2,\ldots, b_r$, with $r= \sum_{i=1}^{\Delta} (4i-2)$ and $\mathcal{L}$ consists of 
\begin{itemize}
    \item the inequality system given in equations~\ref{eq:df-stubstar}~and~\ref{eq:df-stubstar-non-negativity},
    \item for all $0< i< j \le \Delta$, the equation
 \begin{equation}  
     \sum_{\lambda| h(\lambda) = i} x_{\lambda} N(\lambda,j) = \sum_{\lambda'| h(\lambda') = j} x_{\lambda'} N(\lambda',i), \label{eq:lmg-forced-bipartite}
 \end{equation}
     \item  and for all $i$, the equation
\begin{equation}
    \sum_{\lambda | h(\lambda) = i}x_{\lambda}N(\lambda, i) = 2p_i.\label{eq:lmg-parity}
    \end{equation}
     the integer linear programming feasibility problem defining
     \begin{equation}
         z_{s(i)} = \max_{k\in\{1,2,\ldots,i\}}\left\{\min\left\{1,\sum_{\lambda|h(\lambda) = i \wedge N(\lambda,i) = k}x_\lambda\right\}k\right\}, \label{eq:max-degree-lp}
     \end{equation}
     where $s(i)$ is the index of the auxiliary variable storing the maximum degree in the $(i,i)^{\mathrm{th}}$ row of the special color-degree matrix as well as the inequality
\begin{equation}
    2z_{s(i)} \le \sum_{\lambda | h(\lambda) = i}x_{\lambda}N(\lambda, i)\label{eq:lmg-fesability}
\end{equation}

\end{itemize}
\end{theorem}
\begin{proof}
Similar to the proof of Theorem~\ref{theo:mg-realization}. Observe that the linear programming described in equations~\ref{eq:max-degree-lp}~and~\ref{eq:lmg-fesability} is equivalent with the inequality in equation~\ref{eq:loopless-multigraph-condition}.
\end{proof}

In case of simple graph realizations, the first few Erd{\H o}s-Gallai and Gale-Ryser inequalities should be checked. The challenge is that the degrees in 
$D_{i,j}(\mathcal{S})$ and $D_{i,i}(\mathcal{S})$ are not ordered
so we do not know what the first $k$ highest degrees are. Fortunately, it is possible to give a dynamic programming algorithm that computes the sum of the first $k$ highest degrees.
\begin{lemma}\label{lem:dp-for-maxsum}
Let $\mathcal{S}$ be an ensemble of labeled stub-stars in which the number of labeled stub-stars with labels $\lambda$ is $x_{\lambda}$.
Fix an $i$ and $j$, and let 
$$
s_l := \sum_{\lambda|h(\lambda) = i \wedge N(\lambda,j) = l} x_{\lambda}
$$
Let $t_{l,k}$ denote the largest possible sum of at most $k$ entries in 
$D_{i,j}(\mathcal{S})$
corresponding to labeled stub-stars with degree $i$ such that each term in the sum is at most $l$. Then the following equations hold:
\begin{eqnarray}
t_{l,0} & = & 0 \label{eq:dp-for-maxsum-init}\\
t_{l,k} & = & \max_{0\le k' \le k}\left\{t_{l-1,k-k'}+l\times\min\{k',s_l\}\right\}\,\,\,\,\, \forall\ l>1 \label{eq:dp-for-maxsum}
\end{eqnarray}
\end{lemma}
\begin{proof}
Equation~\ref{eq:dp-for-maxsum-init} is trivial. We prove equation~\ref{eq:dp-for-maxsum} by proving two inequalities. To prove that the left hand side is smaller than or equal the right hand side, consider the sum that maximizes $t_{l,k}$. Assume that it contains $\tilde{k}$ terms that are smaller than $l$. Then
$$
t_{l,k} \le t_{l-1,k-\tilde{k}}+l\times\min\{\tilde{k},s_l\}
$$
since the sum of the other $k-\tilde{k}$ terms cannot be larger than $t_{l-1,k-\tilde{k}}$. Thus,
$$
t_{l,k} \le t_{l-1,k-\tilde{k}}+l\times\min\{\tilde{k},s_l\}\le \max_{0\le k' \le k}\left\{t_{l-1,k-k'}+l\times\min\{k',s_l\}\right\}.
$$
To prove that the left hand side is greater than or equal the right hand side, consider the $\tilde{k}$ that maximizes the right hand side, and consider the sum that maximizes $t_{l-1,k-\tilde{k}}$. Add $\min\{\tilde{k},s_l\}$ times $l$ to this sum, then we get a sum that contains at most $k$ terms, the largest term is at most $l$, and its value is $t_{l-1,k-k'}+l\times\min\{k',s_l\}$. It cannot be larger than $t_{l,k}$, thus we get that
$$
t_{l,k} \ge t_{l-1,k-\tilde{k}}+l\times\min\{\tilde{k},s_l\}= \max_{0\le k' \le k}\left\{t_{l-1,k-k'}+l\times\min\{k',s_l\}\right\}.
$$
\end{proof}

For sake of readability, we omitted indexes $i$ and $j$.
We need the same dynamic programming algorithm for each $D_{i,j}(\mathcal(S))$, and we will denote the variables in equation~\ref{eq:dp-for-maxsum} by $t_{l,k}^{i,j}$.

In the Erd{\H o}s-Gallai inequalities, we also need to compute $\sum_{h=k+1}^n \min\{k,d_h\}$. We can define the degree sequence $\tilde{D}$ as $\tilde{d}_h := \min\{k,d_h\}$, and then 
$$
\sum_{h=k+1}^n \min\{k,d_h\} = \sum_{g=1}^n \tilde{d}_g - \tilde{t}_{k}
$$
where $\tilde{t}_k$ is the sum of the at most $k$ larges degrees in $\tilde{D}$. This can be obtained by the dynamic programming recursion
\begin{eqnarray}
\tilde{t}_{l,0} & = & 0 \label{eq:tdp-for-maxsum-init}\\
\tilde{t}_{l,k} & = & \max_{0\le k' \le k}\left\{\tilde{t}_{l-1,k-k'}+\min\{l,k\}\times\min\{k',s_l\}\right\}\,\,\,\,\, \forall\ l>1, \label{eq:tdp-for-maxsum}
\end{eqnarray}
and then defining $\tilde{t}_k := \tilde{t}_{i,k}$
The correctness of the equations~\ref{eq:tdp-for-maxsum-init}~and~\ref{eq:tdp-for-maxsum} can be proved similarly to the proof of Lemma~\ref{lem:dp-for-maxsum}. We need the same dynamic programming algorithm for each $D_{i,i}(\mathcal{S})$, and we will denote the variables in the equation~\ref{eq:tdp-for-maxsum} by $\tilde{t}^i_{l,k}$.

Now we are ready to prove the theorem on simple graph realizations.
\begin{theorem}
Let $D= d_1, d_2,\ldots d_n$ and $F = f_1, f_2, \ldots, f_n$ be a pair of degree and neighbor degree sequences. Let $\Delta$ denote the largest degree in $D$. Then $(D,F)$ is simple graph realizable if and only if the integer programming feasibility problem $\{\mathcal{X}, \mathcal{L}\}$ has a solution, where $\mathcal{X}$ consists of all $x_{\lambda}$ variables for all partitions  $\lambda = \lambda_1, \lambda_2,\ldots, \lambda_{h(\lambda)}$ such that $h(\lambda) \le D$ and for all $i$, $\lambda_i \le \Delta$ and auxiliary variables $p_1, p_2, \ldots, p_{\Delta}, t_{l,k}^{i,j}$ for all $1\le i\le j\le \Delta$ and $1\le k,l \le \Delta$,  $\tilde{t}_{l,k}^{i}$ for all $1\le i\le \Delta$ and $1\le k,l \le \Delta$, $z_1, z_2, \ldots, z_r, b_1, b_2,\ldots, b_r$, with $r= O(\Delta^5)$ and $\mathcal{L}$ consists of 
\begin{itemize}
    \item the inequality system given in equations~\ref{eq:df-stubstar}~and~\ref{eq:df-stubstar-non-negativity},
    
    \item for all $1\le i\le j\le \Delta$, the integer linear programming feasibility problem defining the dynamic programming recursions in equation~\ref{eq:dp-for-maxsum}, and for all $i$, the integer linear programming feasibility problem defining the dynamic programming recursions in equation~\ref{eq:tdp-for-maxsum},
    
    \item for all $1\le i< j \le \Delta$, the equation
 \begin{equation}  
     \sum_{\lambda| h(\lambda) = i} x_{\lambda} N(\lambda,j) = \sum_{\lambda'| h(\lambda') = j} x_{\lambda'} N(\lambda',i), \label{eq:sg-forced-bipartite}
 \end{equation}
     as well as the linear inequality system for the first $\Delta-1$ Gale-Ryser inequalities presented as
     \begin{equation}
         t_{i,k}^{i,j} \le \sum_{\lambda |h(\lambda) = j} x_{\lambda} \min\{k,N(\lambda,i)\}
     \end{equation}

     \item  and for all $i$, the equation
\begin{equation}
    \sum_{\lambda | h(\lambda) = i}x_{\lambda}N(\lambda, i) = 2p_i.\label{eq:sg-parity}
    \end{equation}
   as well as the first $\Delta$ Erd{\H o}s-Gallai inequalities expressed as
   \begin{equation}
       t_{i,k}^{i} \le k(k+1) + \left(\sum_{\lambda | h(\lambda) = i} x_{\lambda} \min\{k,N(\lambda,i)\}\right) - \tilde{t}_{i,k}^{i}.
   \end{equation}
\end{itemize}
\end{theorem}
\begin{proof}
Similar to the proof of Theorem~\ref{theo:mg-realization}.
\end{proof}

The theorem for tree realizations do not need inequalities with minimums and/or maximums.

\begin{theorem}
Let $D= d_1, d_2,\ldots d_n$ and $F = f_1, f_2, \ldots, f_n$ be a pair of degree and neighbor degree sequences. Let $\Delta$ denote the largest degree in $D$. Then $(D,F)$ is tree realizable if and only if the integer programming feasibility problem $\{\mathcal{X}, \mathcal{L}\}$ has a solution, where $\mathcal{X}$ consists of all $x_{\lambda}$ variables for all partitions  $\lambda = \lambda_1, \lambda_2,\ldots, \lambda_{h(\lambda)}$ such that $h(\lambda) \le D$ and for all $i$, $\lambda_i \le \Delta$ and auxiliary variables $p_1, p_2, \ldots, p_{\Delta}$ and $\mathcal{L}$ consists of 
\begin{itemize}
    \item the inequality system given in equations~\ref{eq:df-stubstar}~and~\ref{eq:df-stubstar-non-negativity},

    \item for all $1\le i< j \le \Delta$, the equation
 \begin{equation}  
     \sum_{\lambda| h(\lambda) = i} x_{\lambda} N(\lambda,j) = \sum_{\lambda'| h(\lambda') = j} x_{\lambda'} N(\lambda',i), \label{eq:t-forced-bipartite}
 \end{equation}

     \item  for all $i$, the equation
\begin{equation}
    \sum_{\lambda | h(\lambda) = i}x_{\lambda}N(\lambda, i) = 2p_i.\label{eq:t-parity}
    \end{equation}
  
     item and for all subsets $\mathcal{I} \subseteq \{(i,j) | 1\le i\le j\le \Delta \}$, the inequality
     \begin{eqnarray}
         &&\sum_{(i,j) \in \mathcal{I}|i\neq j}\left( \sum_{\lambda | h(\lambda) = i} x_{\lambda} N(\lambda,j) + \sum_{\lambda | h(\lambda) = j} x_{\lambda} N(\lambda,i) \right)  +
         \sum_{(i,i) \in \mathcal{I}} \sum_{\lambda | h(\lambda) = i} x_{\lambda} N(\lambda,i)
         \le \nonumber 
         \\&&2 
        \sum_{\lambda} x_{\lambda} \min\left\{1,\sum_{j|(h(\lambda),j) \in \mathcal{I}}N(\lambda,j)\right\} 
        %
         - 2. 
     \end{eqnarray}
     
\end{itemize}
Furthermore, $(D,F)$ has a caterpillar realization if a solution exists with
\begin{equation}
    x_{\lambda} = 0
\end{equation}
for all $\lambda$ for which $h(\lambda) >1$ and $\sum_{j>1} N(\lambda,j) \le 2$.
\end{theorem}
\begin{proof}
Similar to the proof of Theorem~\ref{theo:mg-realization}.
\end{proof}

%

\section{Simulation results}

In order to demonstrate the practical implications of the presented results, we devised a set of experiments which show the potential of our approach in certain types of applications.
Furthermore, we were also able to show where the limitations of this approach were. 
Since the original motivation for this works comes from chemistry, we considered evaluating the reconstruction of saturated hydrocarbons.

We wanted to see two main results 1) how fast we can find one solution (i.e., one tree from a given input sequence) and 2) how fast the number of solutions grows with the size of the input.
We constructed two different testsets for each of these two goals.
For the first goal we generated trees of sizes 100 to 1000 with a step of 100. For the second goal we generated testsets of sizes only from 10 to 80 with a step of 10 (for reasons that will be obvious from the results). For all sizes we generated 100 trees and the obtained results are all averages (average times and average counts) for these 100 measurements.

Each tree is generated with a simple recursive procedure described below. The procedure generates a random tree with maximal degree $k$ such that for the current vertex,
\begin{enumerate}
    \item it generates a random degree $d$ from $[1\ldots k-1]$ (except for the first node where the interval is $[1\ldots k]$) uniformly at random,
    \item it chooses a random partition of $n-1$ into $d$ pieces, i.e., choose how large each of the $d$ subtrees will be ,
    \item then it recursively repeats the process for all $d$ neighbors until reaching the leaves.
\end{enumerate}

In order to harness the best practices in integer programming, we used a state-of-the-art solver and used it to construct feasible color-degree matrices.
Our goal was to use out-of-the-box IP solvers and evaluate its performance by varying the size of the problem. Since in chemical applications the input represents a feasible tree (i.e. we know there exists a feasible solution) we also start with a tree, decompose it into a degree sequence and a sequence of sums of neighbor degrees.

 We used an openly accessible SCIP solver~\cite{scip} and in order to generate a specific integer program we used the ZIMPL language~\cite{zimpl}. The final trees were reconstructed with the custom written program in python.


The main two results are presented in Figure~\ref{fig:timing} and Figure~\ref{fig:numsols}.  Figure~\ref{fig:timing} shows that the time required to find one feasible solution does not increase in the range that we tested it. It is not surprising that neither the number of variables nor the number of inequalities in the integer programming depends on the input size. In the indicated range of input size, most of the time is spent on solving the integer programming problem and only a minor part of the running time of the solver is spent on actually building the tree. 

Figure~\ref{fig:numsols} shows the exponential growth of the average number of feasible solutions to the given input. 
Already with $n=30$ carbon atoms, the average number of solutions to a random input approaches $10$ and with $n=40$ carbon atoms, the average number of solutions is about $86$. The average number of solutions keeps growing exponentially with the number of carbon atoms, indicating that it is not feasible to reconstruct large hydrocarbons using only NMR data and then testing each potential solution.

In chemistry however, besides the information about the neighborhood, other information could be used. Some of those chemical and physical limitations also infer various structural limitations which need to be further explored. 
Namely, completely unrestricted trees yield chemically unfeasible compounds so further work needs to be focused on trees with special structures so that the exponential growth of the number of solutions is postponed even further thus making this approach practical for even larger compounds.

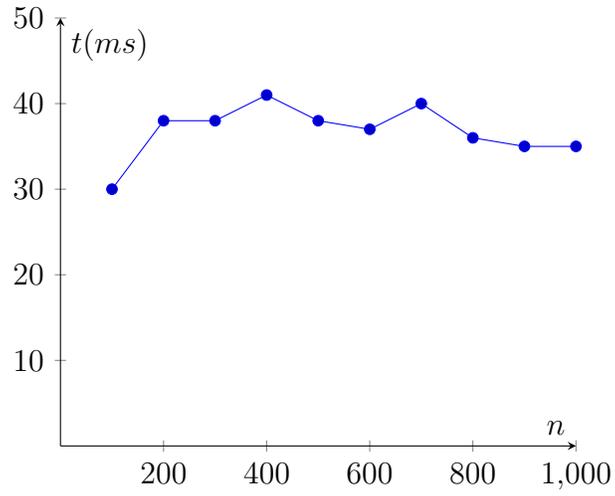
\begin{figure}
    \centering
    \begin{tikzpicture}
\begin{axis}[
    axis lines = middle,
    xlabel = {$n$},
    ylabel = {$t(ms)$},
    xmin=0, xmax=1000,
    ymin=0, ymax=50]
 
\addplot coordinates {
( 100, 30 )
( 200, 38 )
( 300, 38 )
 ( 400, 41 )
 ( 500, 38 )
 ( 600, 37 )
 ( 700, 40 )
 ( 800, 36 )
 ( 900, 35 )
 ( 1000, 35 )
};
\end{axis}  
\end{tikzpicture}

    \caption{The average time ($t$, measured in milliseconds) needed to generate one feasible solution of a tree with $n$ vertices, maximum degree $4$ and prescribed degree and sum of neighbor degree sequences. See text for detail about generating input trees and finding a solution.}
    \label{fig:timing}
\end{figure}

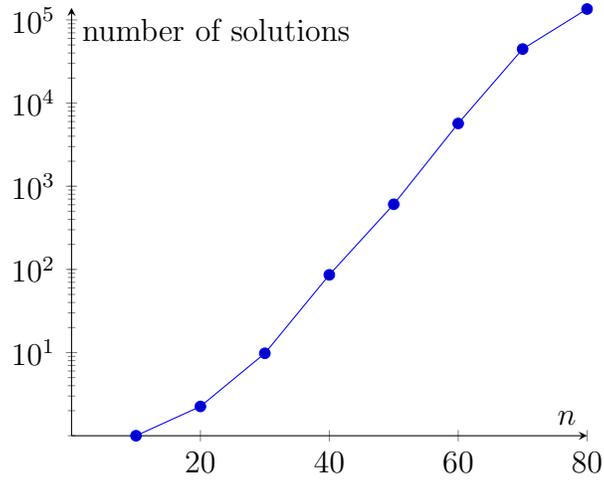
\begin{figure}
    \centering
    \begin{tikzpicture}
        \begin{axis}[
        axis lines = middle,
        ymode=log,
        xlabel = {$n$},
        ylabel = {number of solutions},
        xmin=0, xmax=80,
        ymin=0, ymax=140000]

        \addplot coordinates {
        ( 10, 1 )
        ( 20, 2.25 )
        ( 30, 9.8 )
        ( 40, 86.05 )
        ( 50, 607 )
        ( 60, 5686 )
        ( 70, 44676 )
        ( 80, 135289 )
        };
        \end{axis}  
    \end{tikzpicture}

    \caption{The average number of solutions of a (degree, sum of neighbor degree) sequence of length $n$. See text for detail.}
    \label{fig:numsols}
\end{figure}


\section{Discussion}

In this paper, we considered the problem of constructing a bounded degree graph with prescribed degree and neighbor degree sequence. The key is to find an appropriate ensemble of labeled stub-stars that can realize the given degree and neighbor degree sequence. We talked about monochromatic subgraphs that considers only stubs with given labels incident to vertices with given degrees. This approach highly resembles to the color-degree matrix problem that is also intensively studied \cite{betzetal,bushetal,guinezetal,h-mcdc}.

Indeed, an ensemble of labeled stub-stars, $\mathcal{S}= \{S_1, S_2, \ldots, S_n\}$, can be represented by a \emph{special color-degree matrix} $M$. $M$ has ${\Delta \choose 2} + \Delta$ rows and $n$ columns, where $\Delta$ is the maximum degree in $\mathcal{S}$. The rows are labeled by $(i,j) \in {\Delta\choose 2}$ and $(i,i) \in \Delta^2$. When $(i,j)$ is the label of the $k^{\mathrm{th}}$ row, then 
$$
m_{k,l} = \begin{cases}
\mathrm{number\  of \ } j \mbox{\ labels\ of\ } S_l & \mathrm{if\ } S_l \mathrm{\ has\ degree\ } i  \\
\mathrm{number\  of \ } i \mbox{\ labels\ of\ } S_l & \mathrm{if\ } S_l \mathrm{\ has\ degree\ } j  \\
0 & \mathrm{otherwise}
\end{cases}
$$
We call the $(i,j)$ and $(i,i)$ labels \emph{colors}. We will index the rows of $M$ by their color indexes, that is, we will talk about the $(i,j)^{\mathrm{th}}$ row and we index entries as $m_{(i,j),l}$.

Color-degree matrices appeared earlier in the scientific literature, see for example \cite{h-mcdc}. The color-degree matrices introduced here are special in two ways: first, for any two rows with indexes $(i,j)$ and $(i',j')$ such that $i\notin \{i',j'\}$ and $j\notin \{i',j'\}$ and for any column index $l$, $m_{(i,j),l} \neq 0$ implies that $m_{(i',j'),l} = 0$. Second, we consider only special realizations of these special color-degree matrices.

Indeed, observe that not all edge disjoint monochromatic realizations of the degree sequences in the rows of $M$ are realizations of the corresponding ensemble of stub-stars. For each $i\neq j$, it is required that the monochromatic realization be a \emph{forced} bipartite realization, that is, we require that each edge goes between prescribed vertex classes. 

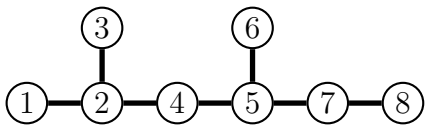
\begin{figure}[htb]
    \centering
    \begin{tikzpicture}
        \node[vertex](1) at (1,0) {1};
        \node[vertex](2) at (2,0) {2};
        \node[vertex](3) at (2,1) {3};
        \node[vertex](4) at (3,0) {4};
        \node[vertex](5) at (4,0) {5};
        \node[vertex](6) at (4,1) {6};
        \node[vertex](7) at (5,0) {7};
        \node[vertex](8) at (6,0) {8};
        \draw[line width = 2pt] (1) --(2) --(4) --(5)--(7) --(8) (2)--(3) (5)--(6);
    \end{tikzpicture}
        \caption{Caterpillar counterexample $C_1$. See text for details.}
        \label{fig:3232cat}
\end{figure}

To see why this is necessary, consider the caterpillar $C_1$, shown in Figure~\ref{fig:3232cat}. Its color-degree matrix will have the form 
\[
M=
\begin{matrix}
$(1,2)$\\
$(1,3)$\\
$(2,3)$\\
$(1,1)$\\
$(2,2)$\\
$(3,3)$
\end{matrix}
\begin{pmatrix}
0&0&0&0&0&0&1&1\\
1&2&1& 0 & 1 & 1 & 0 & 0\\
0&1&0&2&2&0&1&0\\
0&0&0&0&0&0&0&0\\
0&0&0&0&0&0&0&0\\
0&0&0&0&0&0&0&0
\end{pmatrix}.
\]
Notice that the 
plain, dashed, dotted subgraphs in Figure~\ref{fig:3322cat}, respectively, are edge disjoint realizations for the $(1,2)$-th, $(1,3)$-th, and $(2,3)$-th rows of $M$, respectively. However, their union is $C_2$ which is not a realization of $M$ because $C_2$ has $(2,2)$ and $(3,3)$ edges. Also observe that the dotted subgraph in Figure~\ref{fig:3322cat} is a bipartite graph, however, it is not a forced bipartite graph.
\begin{figure}[htb]
    \centering
    \begin{tikzpicture}
        \node[vertex](1) at (1,0) {1};
        \node[vertex](2) at (2,0) {2};
        \node[vertex](3) at (2,1) {3};
        \node[vertex](4) at (4,0) {4};
        \node[vertex](5) at (3,0) {5};
        \node[vertex](6) at (3,1) {6};
        \node[vertex](7) at (5,0) {7};
        \node[vertex](8) at (6,0) {8};
        \draw[line width = 2pt, dashed] (1) --(2) --(3) (5)--(6);
        \draw[line width = 1pt, dotted] (2)--(5)--(4)--(7);
        \draw[line width = 2pt] (7)--(8);
    \end{tikzpicture}
        \caption{Caterpillar counterexample $C_2$.}
        \label{fig:3322cat}
\end{figure}
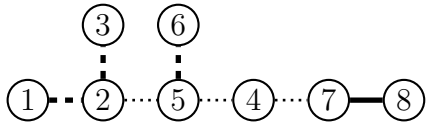

In general, it is NP-complete to decide if a color-degree matrix has any realization \cite{durretal,gardneretal}. On the other hand, Carrol and Isaak showed that the problem is easy if each row if the column sums form a degree sequence of a forest \cite{ci2009}. Hillebrand and McDiarmid showed that the problem remains easy if the column sums form a degree sequence that has a realization with at most one cycle \cite{h-mcdc}. These results cannot be directly applied on our problem since we require special realizations as discussed above.

Nonetheless, it is worth mentioning that the hard part of the degree and neighbor degree realization problem is to find the appropriate ensemble of stub-stars. Indeed, it remains easy to decide if a special color-degree matrix has a simple graph realization even if there is no bound on the maximal degree: due to the speciality of the color-degree matrices obtained from ensemble of stub-stars, it is guaranteed that the union of monochromatic simple graph realizations will remain simple graphs (given that the realizations for $(i,j)$ colors, $i\ne j$ are forced bipartite).

\section*{Acknowledgement}
I.M. was supported by NKFIH grants KH126853, K132696 and SNN135643. The project is a continuation of the work done at the 2020 Budapest Semesters in Mathematics. All authors would like to thank the BSM for running the program.

\bibliographystyle{elsarticle-harv}

\begin{thebibliography}{1}

\bibitem{betzetal} C. Bentz, M.-C. Costa, C. Picouleau, B. Ries, D. de Werra, Degree-constrained edge
partitioning in graphs arising from discrete tomography, Journal of Graph Algorithms
and Applications 13 (2) (2009)

\bibitem{bushetal} A. Busch, M. Ferrara, S. Hartke, M. Jacobson, H. Kaul, D. West, Packing of graphic
n-tuples, Journal of Graph Theory 70 (1) (2012)

\bibitem{ci2009} J. Carroll, G. Isaak, Degree matrices realized by edge-colored forests, Online;  (Oct. 2009).
\url{http://www.lehigh.edu/~gi02/ecforest.pdf}

\bibitem{czabarkaetal} Czabarka, \'E., Dutle, A., Erd{\H o}s, P.L., Mikl\'os, I. (2014) On Realizations of a Joint Degree Matrix, Discrete Applied Mathematics, 181(30):283-288.

\bibitem{durretal} C. D\"urr, F. Gu{\'\i}\~nez, M. Matamala, Reconstructing 3-colored grids from horizontal
and vertical projections is NP-hard: A solution to the 2-atom problem in discrete
tomography, SIAM Journal on Discrete Mathematics 26 (1) (2012)

\bibitem{eg} Erd{\H o}s, P., Gallai, T. (1960)  Graphs with vertices of prescribed degrees  (in Hungarian) Matematikai Lapok, 11: 264--274.


\bibitem{em2018} Erd{\H o}s, E.L., Mikl\'os, I. (2018) Not all simple looking degree sequence problems are easy, Journal of Combinatorics, 9(3):553--566.

\bibitem{gale} Gale, D. (1957) A theorem on flows in networks. Pacific J. Math. 7 (2): 1073--1082.


\bibitem{scip} Gamrath G., Anderson D., Bestuzheva K., Chen W., Eifler L., Gasse M., Gemander P., Gleixner A., Gottwald L., Halbig K., Hendel G., Hojny C., Koch T., Le Bodic P., Maher S.J., Matter F., Miltenberger M., M{\"u}hmer E., M{\"u}ller B. Pfetsch M.E., Schl{\"o}sser F., Serrano F. Shinano Y., Tawfik C., Vigerske S., Wegscheider F., Weninger D., Witzig J., (2020) The SCIP Optimization Suite 7.0, Zuse Institute 
Berlin, ZIB-report, 20-10.

\bibitem{zimpl} Koch, Thorsten. "Rapid mathematical programming.", PhD thesis, (2005).

\bibitem{gardneretal} R. Gardner, P. Gritzmann, D. Prangenberg, On the computational complexity of
reconstructing lattice sets from their x-rays, Discrete Mathematics 202 (1–3) (1999)
45 – 71.

\bibitem{guinezetal} F. Gu{\'\i}\~nez, M. Matamala, S. Thomass\'e, Realizing disjoint degree sequences of span at
most two: A tractable discrete tomography problem, Discrete Applied Mathematics
159 (1) (2011) 23 -- 30.

\bibitem{h-mcdc} Hillebrand, A., McDiarmid, C. (2016) Colour degree matrices of graphs with at most one cycle,
Discrete Appl. Math., 209:144--152.


\bibitem{orsinietal} Orsini, C., Dankulov, M., Colomer-de-Sim\'on, P. et al.(2015) Quantifying randomness in real networks, Nat. Commun 6:8627. 

\bibitem{ryser} Ryser, H. J. (1957) Combinatorial properties of matrices of zeros and ones. Can. J. Math. 9: 371--377.

\bibitem{stantonpinar} Stanton, I., Pinar, A. (2012) Constructing and sampling graphs with a prescribed joint degree distribution, ACM Journal of Experimental Algorithmics, September 17:3.5

\end{thebibliography}

\end{document}